\documentclass[a4paper,12pt]{amsart}
\usepackage[dvipdfmx]{graphicx}
\usepackage{latexsym,amssymb,amsmath,mathrsfs}
\usepackage{amsthm}
\usepackage{mathrsfs}
\usepackage{color}
\usepackage[english]{babel}
\usepackage{comment}
\usepackage[abbrev]{amsrefs}
\usepackage[top=30truemm,bottom=30truemm,left=20truemm,right=20truemm]{geometry}
\newtheorem{theorem}{Theorem}[section]

\newtheorem{proposition}[theorem]{Proposition}

\theoremstyle{remark}
\newtheorem{remark}[theorem]{Remark}
\DeclareMathOperator{\Ric}{Ric}

\address{Graduate School of Mathematical Sciences, The University of Tokyo, Komaba, Tokyo, 153-8914, Japan} 
\email{yasuakifujitani@g.ecc.u-tokyo.ac.jp}

\newcommand{\e}{\mathrm{e}}
\newcommand{\Hess}{\mathrm{Hess}\,}
\renewcommand{\d}{\mathrm{d}}

\makeatletter
\@namedef{subjclassname@2020}{\textup{2020} Mathematics Subject Classification}
\makeatother

\title{Choi-Wang inequality for affine connections}
\author{Yasuaki Fujitani}
\begin{document}
\maketitle
\begin{abstract}
    Choi-Wang established a lower bound for the first \textcolor{black}{non-zero} eigenvalue of the Laplacian on minimal hypersurfaces in manifolds with positive Ricci curvature.
    We extend this Choi-Wang type inequality to the setting of positive Ricci curvature with respect to the Li-Xia type affine connection.
\end{abstract}
\section{Introduction}
Let $(M,g)$ be an $n$-dimensional Riemannian manifold.
For $u \in C^{\infty}(M)$,
$\alpha,\beta \in \mathbb{R}$ and the Levi-Civita connection $\nabla$ for $g$,
Li-Xia \cite{LX} introduced the following affine connection:
\begin{align}\label{eq:LX}
    \nabla^{u,\alpha,\beta}_XY = \nabla_XY + \alpha \d {u} (X)Y + \alpha \d {u}(Y) X + \beta g(X,Y)\nabla {u}.
\end{align}
We refer to this as the \textit{Li-Xia type affine connection}.
We call the Ricci curvature with respect to $\nabla^{u,\alpha,\beta}$ by the \textit{Li-Xia type affine Ricci curvature} and denote it by $\Ric^D$ with $D := \nabla^{u,\alpha,\beta}$.
In \cite{LX},
they gave a Lichnerowicz type inequality for manifolds satisfying
\begin{align}\label{eq:LX-condition}
    \Ric^D \geq K \e^{(\alpha- \beta)u}g
\end{align}
with $K > 0$.
In this paper,
we also work under this condition, 
which serves as a bridge in relating various geometric frameworks.
Indeed, 
the Li-Xia type affine Ricci curvature provides a unified perspective that connects geometry of manifolds with the substatic condition and geometry of manifolds with $1$-weighted Ricci curvature bounded from below.

To explore this connection in more detail,
we begin with providing a brief introduction of the substatic condition.
For positive $V \in C^{\infty}(M)$,
the substatic condition is given by
\begin{align*}
    V \Ric - \Hess V + (\Delta V) g \geq 0.
\end{align*}
For instance, 
the deSitter-Schwarzschild manifolds and the Reissner-Nordstr\"{o}m manifolds satisfy this condition.
In general relativity,
this condition is related to the Null energy condition.
Recently,
Brendle \cite{brendle2013constant} derived an Alexandrov type theorem for a class of warped product manifolds,
which satisfies the substatic condition.
Furthermore,
Borghini-Fogagnolo \cite{BF} obtained a Bishop-Gromov type volume comparison theorem for manifolds with the substatic condition.
Regarding the relation with $\nabla^{u,\alpha,\beta}$,
in the case $\alpha = 0$ and $\beta = 1$,
Li-Xia type affine Ricci curvature is equivalent to the static Ricci tensor:
\begin{align*}
    \Ric - \frac{\Hess V}{V} + \frac{\Delta V}{V} g
\end{align*}
with $V = \e^u$.
Hence,
we see that the condition of non-negative static Ricci tensor is identical to the substatic condition.
The curvature bound \eqref{eq:LX-condition} coincides with the condition:
\begin{align*}
    V \Ric - \Hess V + (\Delta V)g \geq K g.
\end{align*}

We proceed to introduce the $1$-weighted Ricci curvature.  
We consider a weighted Riemannian manifold $(M,g,\e^{-f}v_g)$,
where $f \in C^{\infty}(M)$ and $v_g$ is the Riemannian volume measure.
For $N \in (-\infty, 1] \cup [n, \infty]$,
the \textit{$N$-weighted Ricci curvature} is defined by
\begin{align*}
    \Ric_f^N = \Ric + \Hess f - \frac{\d f\otimes \d f}{N-n}.
\end{align*}
The last term vanishes when $N = \infty$ and we consider only a constant function as $f$ when $N = n$.
Moreover,
for $N \in (n,\infty)$ and $N' \in (-\infty,1)$,
we have the relation:
\begin{align*}
    \Ric_f^N  \leq \Ric_f^\infty \leq \Ric_f^{N'} \leq \Ric_f^1.
\end{align*}
Therefore,
we can consider the condition $\Ric_f^1\geq K g$ to be the weakest among the conditions $\Ric_f^N \geq K g$ for various $N \in (-\infty,1]\cup [n,\infty]$.
In comparison geometry,
research on manifolds with Ricci curvature bounded from below has been extended to weighted Riemannian manifolds with $N$-weighted Ricci curvature bounded from below.
For example,
Wei-Wylie \cite{WW} derived a Bishop-Gromov type volume comparison theorem in the case $N = \infty$.
Later,
Wylie-Yeroshkin generalized it to the case $N = 1$.
We note that any Bishop-Gromov type volume comparison theorem has not yet been obtained under the condition \eqref{eq:LX-condition}.
In particular, 
for weighted Riemannian manifolds,  
the curvature assumption in \eqref{eq:LX-condition} is equivalent to
\begin{align*}
    \Ric_f^1 \geq K\e^{\frac{-f}{n-1}}g,
\end{align*}
and this is different from the setting considered by Wylie-Yeroshkin~\cite{WY}.
The case $N = 1$ also naturally arises in the context of affine connections.  
Indeed, 
when $\alpha = (n-1)^{-1}$, 
$\beta = 0$ and $u \equiv -f$, 
the Li-Xia type affine connection reduces to the one introduced earlier by Wylie-Yeroshkin \cite{WY}, 
where they revealed that the Ricci curvature with respect to this connection is equivalent to the $1$-weighted Ricci curvature.

We next turn to the Choi-Wang inequality.
Choi-Wang \cite{CW} was the first to establish a lower bound for the first \textcolor{black}{non-zero} eigenvalue of the Laplacian on minimal hypersurfaces in an orientable closed Riemannian manifold with positive Ricci curvature.
Their proof relied on the first Betti number estimate obtained by Bochner \cite{bochner1948curvature}.
Subsequently,
Choi-Schoen \cite{CS} extended the Choi-Wang type inequality to a closed Riemannian manifold with positive Ricci curvature.
Instead of using the Bochner theorem,
they used the finiteness of the fundamental group,
which allowed them to drop the assumption of orientability.
Using their Choi-Wang type inequality,
Choi-Schoen \cite{CS} showed a compactness theorem for minimal surfaces.
\textcolor{black}{These results were generalized to weighted settings.
In particular,
the \textit{weighted Laplacian} on the weighted Riemannian manifold $(M,g,\e^{-f}v_g)$ is defined by 
\begin{align*}
    \Delta_f \varphi := \Delta \varphi - g(\nabla f, \nabla \varphi)
\end{align*}
for $\varphi \in C^\infty(M)$.
In weighted Riemannian manifolds with positive $\Ric_f^N$,
Choi-Wang type inequalities for the weighted Laplacian have been investigated (see e.g., \cite{MD, LW, CMZ, FS}).
}
As far as we know,
the Choi-Wang type inequality has not been extended to the case $N = 1$.
\textcolor{black}{
In this paper,
we generalize the Choi-Wang inequality in \cite{CW} to manifolds satisfying \eqref{eq:LX-condition}.
We state our main theorem.
We note that notations and terminologies in the following statement are introduced later in section \ref{sec:Choi-Wang}.
Our main theorem is as follows:
\begin{theorem}\label{thm:choi-wang}
Let $(M,g)$ be a compact orientable Riemannian manifold, $u\in C^{\infty}(M)$ and $\alpha,\beta \in \mathbb{R}$.
We set $D:= \nabla^{u,\alpha,\beta}$.
For $K > 0$,
we assume 
\begin{align*}
    \Ric^D \geq K \e^{(\alpha  - \beta)u} g.
\end{align*}
Let $\Sigma$ be a compact orientable embedded $D$-minimal hypersurface in $M$.
Then the first \textcolor{black}{non-zero} eigenvalue $\lambda_{1}(\Delta^D_\Sigma)$ of the $D$-Laplacian $\Delta^D_{\Sigma}$ on $\Sigma$ satisfies 
\begin{align*}
    \lambda_{1}(\Delta^D_\Sigma) \geq \frac{K}{2}.
\end{align*}
\end{theorem}}

A major challenge in extending the argument of \cite{CW} to our setting is that the Bochner type theorem is not directly applicable.  
To overcome this obstacle, 
we introduce a new structure on the Li-Xia type affine connection from the viewpoint of the statistical manifold.
We first explain this structure in detail.  
In the framework of affine differential geometry, 
a statistical manifold is a Riemannian manifold equipped with a torsion free affine connection which satisfies the Codazzi equation. 
Such structures have long been known in the theory of affine hypersurfaces. 
In the context of statistics, 
these structures were introduced by Lauritzen \cite{lauritzen1987statistical} after pioneering works by Amari \cite{amari1982differential,amari1985differential}.
For weighted Riemannian manifolds, 
it was shown by Wylie-Yeroshkin \cite{WY} that the torsion free affine connection introduced by themselves satisfies the Codazzi equation, 
and a more detailed study was given by Yeroshkin \cite{Yero}.
Our statistical structure on Li-Xia type affine connections can be regarded as a natural generalization of the one presented in \cite{WY, Yero}.
Using this new structure,
we are then able to establish the Bochner type theorem under our assumption,
which in turn allows us to prove the Choi-Wang type inequality.
Indeed,
after introducing the statistical structure on Li-Xia type affine connections,
we apply the Bochner type theorem for statistical manifolds obtained by Opozda \cite{O}.
As a consequence,
we obtain a Bochner type theorem for Riemannian manifolds with non-negative Li-Xia type affine Ricci curvature. 
Moreover,
we note that our statistical structure allows us to derive a Lichnerowicz type inequality in our setting via the argument in \cite{O}, 
and it turns out that this provides an alternative proof of a Lichnerowicz type inequality originally shown in Li-Xia \cite{LX}.

Although the Choi-Wang inequality was generalized by Choi-Schoen \cite{CS} to more general manifolds,  
their argument does not directly apply to the setting we consider in this paper.
One key obstruction is that the finiteness of the fundamental group, 
which is assumed in their argument, 
is not known to hold in our situation.
To establish the finiteness of the fundamental group,
it would be necessary to verify the Bishop-Gromov type volume comparison theorem for manifolds satisfying \eqref{eq:LX-condition}.  
We regard this as a direction for future research.

\textcolor{black}{This paper is organized as follows.}
In section \ref{sec:bochner},
we show a statistical structure of Li-Xia type affine connection,
and obtain the Bochner type theorem \textcolor{black}{(Theorem \ref{thm:bochner})}.
In section \ref{sec:Choi-Wang},
\textcolor{black}{
after introducing the notations and terminologies in Theorem \ref{thm:choi-wang},
we give the proof.
}

\section{Bochner type theorem}\label{sec:bochner}
Let $(M,g)$ be an $n$-dimensional Riemannian manifold.
In this paper,
we only consider connected manifolds without boundary.
For a torsion-free affine connection $D$ on $M$,
the \textit{dual connection} $D^*$ for $(M,g,D)$ is defined by the following relation:
\begin{align*}
    X g(Y,Z) = g(D_X Y, Z) + g(Y, D^*_X Z).
\end{align*}
If the \textit{Amari-Chentsov tensor} $C(X,Y,Z) := (D_X g)(Y,Z)$ is symmetric,
$(M,g,D)$ is called \textit{statistical manifold}.
The \textit{Riemannian curvature operator $R^D$} and the \textit{Ricci curvature $\Ric^{D}$ with respect to $D$} are defined as follows: 
\begin{align}
    R^D(X, Y) Z &:= D_X D_Y Z-D_Y D_X Z-D_{[X, Y]} Z, \nonumber\\
    \Ric^D(X,Y) &:= \sum_{i = 1}^n g( R^D(E_i,X)Y,E_i )\label{eq:relation-riem-ric},
\end{align}
where $\{E_1, \dots, E_n\}$ is an orthonormal frame of the tangent bundle.
We have the following duality relation:
\begin{proposition}\label{prop:dual}
Let $(M,g)$ be a Riemannian manifold, ${u}\in C^{\infty}(M)$ and $\alpha,\beta \in \mathbb{R}$.
We set $\tilde{g}:= \e^{2(\alpha-\beta){u}}g$.
We denote the Levi-Civita connection for $g$ and $\tilde{g}$ by $\nabla$ and $\widetilde{\nabla}$, 
respectively.
Then $\widetilde{\nabla}^{-u,\alpha,\beta}$ is the dual connection for $(M,\e^{(\alpha-\beta) {u}}g,\nabla^{u,\alpha,\beta})$, i.e.,
we have 
\begin{align*}
    X (\e^{(\alpha-\beta) {u}}g(Y,Z)) = \e^{(\alpha-\beta) {u}}g(\nabla^{u,\alpha,\beta}_X Y, Z) + \e^{(\alpha-\beta) {u}}g(Y,\widetilde{\nabla}^{-u,\alpha,\beta}_X Z)
\end{align*}
for any vector fields $X,Y,Z$.
\end{proposition}
\begin{proof}
Since 
\begin{align*}
    \widetilde{\nabla}_X Z = \nabla_X Z + (\alpha - \beta)\d{u}(X)Z +  (\alpha - \beta)\d{u}(Z)X - (\alpha - \beta)g(X,Z)\nabla{u},
\end{align*}
we see
\begin{align*}
    \widetilde{\nabla}^{-u,\alpha,\beta}_X Z = \nabla_X Z - \beta \d {u}(X)Z - \beta \d {u}(Z)X - \alpha g(X,Z)\nabla {u}.
\end{align*}
For $\overline{g} := \e^{(\alpha -\beta)u}g$,
we obtain 
\begin{align*}
    &\overline{g}(\nabla^{u,\alpha,\beta}_X Y, Z) + \overline{g}(Y,\widetilde{\nabla}^{-u,\alpha,\beta}_X Z)\\
    &= \overline{g}\left( \nabla_XY + \alpha \d {u}(X)Y + \alpha \d {u}(Y)X + \beta g(X,Y)\nabla {u},Z \right)\\
    &\,\,\quad + \overline{g}\left( Y, \nabla_X Z  - \beta \d {u}(X)Z - \beta \d {u}(Z)X - \alpha g(X,Z)\nabla {u}\right)\\
    &= \overline{g}(\nabla_XY, Z) +\overline{g}(Y,\nabla_XZ) + (\alpha - \beta)\d{u}(X)\overline{g} (Y,Z)\\
    &= X \overline{g}(Y,Z).
\end{align*}
\end{proof}
\begin{remark}
We refer to Yeroshkin \cite[Proposition 2.4]{Yero} for the case $u \equiv -f$,
$\alpha = (n-1)^{-1}$ and $\beta = 0$ on a weighted Riemannian manifold $(M,g,\e^{-f}v_g)$.
We also note that this result in \cite{Yero} has a further application in \cite[Proposition 2.4]{fujitani2025steklov}.
\end{remark}
Also,  
we possess the following property:
\begin{proposition}
Let $(M,g)$ be a Riemannian manifold, ${u}\in C^{\infty}(M)$ and $\alpha,\beta \in \mathbb{R}$.
We denote the Levi-Civita connection for $g$ by $\nabla$.
Then $(M,\e^{(\alpha - \beta){u}}g,\nabla^{u,\alpha,\beta})$ is a statistical manifold.
\end{proposition}
\begin{proof}
We set $\overline{g} := \e^{(\alpha - \beta)u}g$,
$D := \nabla^{u,\alpha,\beta}$ and $D^* := \widetilde{\nabla}^{-u,\alpha,\beta}$,
where we use the same notations as in Proposition \ref{prop:dual}.
Using the dualistic structure obtained in Proposition \ref{prop:dual},
we have 
\begin{align*}
    (D_X \overline{g})(Y,Z) &= X \overline{g}(Y,Z) - \overline{g}(D_X Y, Z) - \overline{g}(Y, D_X Z)\\
    &= \overline{g}(D_X Y, Z) + \overline{g}(Y, D^*_X Z) - \overline{g}(D_X Y, Z) - \overline{g}(Y, D_X Z)\\
    &= \overline{g}(Y, D^*_X Z) - \overline{g}(Y, D_X Z)\\
    &= \overline{g}(Y, \nabla_X Z - \beta \d u(X) Z - \beta \d u(Z)X - \alpha g (X,Z)\nabla u)\\
    &\,\,\quad -\overline{g}(Y, \nabla_X Z + \alpha \d u(X)Z + \alpha \d u(Z)X + \beta g (X,Z)\nabla u)\\
    &= -(\alpha + \beta) (\d u(X)\overline{g}(Y,Z) + \d u(Z) \overline{g}(X, Y) + \d u(Y) \overline{g}(X,Z)).
\end{align*}
This implies the Amari-Chentsov tensor $(D_X \overline{g})(Y,Z)$ is symmetric.
\end{proof}
\begin{remark}
We refer to Wylie-Yeroshkin \cite[Proposition 5.30]{WY} and Yeroshkin \cite[Remark 2.2]{Yero} for the case $u \equiv -f$,
$\alpha = (n-1)^{-1}$ and $\beta = 0$ on a weighted Riemannian manifold $(M,g,\e^{-f}v_g)$.
\end{remark}
For a measure $\mu$ on $M$,
a quadruple $(M,g,D,\mu)$ is called \textit{equiaffine} when $D\mu = 0$.
By using an equiaffine structure associated with the Li-Xia type affine connection,
we obtain a Bochner type theorem for the first Betti number $b_1(M)$ as follows:
\begin{theorem}\label{thm:bochner}
Let $(M,g)$ be an $n$-dimensional orientable compact Riemannian manifold,
${u}\in C^{\infty}(M)$ and $\alpha,\beta \in \mathbb{R}$.
We denote the Levi-Civita connection for $g$ by $\nabla$ and $D := \nabla^{u,\alpha, \beta}$.
We assume $\Ric^D \geq 0$.
Then we have $b_1(M) \leq n$.
Suppose additionally that $\Ric^D > 0$ at some point,
then $b_1(M) = 0$.
\end{theorem}
\begin{proof}
We set $V := \e^u$ and $\tau := (n+ 1)\alpha + \beta$.
It follows from Li-Xia \cite[Corollary 3.5]{LX} that 
\begin{align*}
    D \left( V^\tau v_g \right) = 0.
\end{align*}
This implies $(M,V^{\alpha - \beta}g, D, V^\tau v_g)$ is equiaffine.
Therefore,
we may apply Opozda \cite[Theorem 9.6]{O} and we arrive at the desired assertion.
\end{proof}

\begin{remark}
    \begin{enumerate}
        \item[] 
        \vskip.25\baselineskip
        \item[(i)] Together with the relation $\Ric_f^\infty \leq \Ric_f^1$,
        this can be regarded as a generalization of the first Betti number estimate for weighted Riemannian manifolds with non-negative $\Ric_f^\infty$ obtained by Lott \cite[Theorem 1 (3)]{L} if we restrict ourselves to orientable manifolds.
        \vskip.25\baselineskip
        \item[(ii)] Regarding the equiaffine structure in the proof,
        we refer to Wylie-Yeroshkin \cite[Corollary 3.5]{WY} for the case $u \equiv -f$,
        $\alpha = (n-1)^{-1}$ and $\beta = 0$ on a weighted Riemannian manifold $(M,g,\e^{-f}v_g)$.
        \vskip.25\baselineskip
        \item[(iii)] This equiaffine structure gives an alternative proof of the Lichnerowicz type inequality in Li-Xia \cite{LX}.
        Indeed,
        applying Opozda \cite[Theorem 9.9]{O} to this equiaffine structure yields the Lichnerowicz inequality in \cite[Theorem 1.3 (iii, a)]{LX}.
    \end{enumerate}
\end{remark}

\section{\textcolor{black}{Proof of the }Choi-Wang inequality}\label{sec:Choi-Wang}
\textcolor{black}{In this section,
we first introduce the notations and terminologies in Theorem \ref{thm:choi-wang} and recall the Reilly formula in Li-Xia \cite{LX}.}
Let $\nabla$ be the Levi-Civita connection \textcolor{black}{of} $g$.
For ${u} \in C^{\infty}(M)$ and $\alpha,\beta \in \mathbb{R}$,
we set $V := \e^u$ \textcolor{black}{and} $D := \nabla^{u,\alpha,\beta}$.
For $\varphi \in C^{\infty}(M)$,
the \textcolor{black}{$D$-}\textit{gradient $\nabla^D$}, the \textcolor{black}{$D$-}\textit{Hessian $\Hess^D$} and  the \textcolor{black}{$D$-}\textit{Laplacian $\Delta^D$} are defined as follows (see \cite[Definition 2.4]{LX}):
\begin{align*}
    {\nabla}^D \varphi &:=V^{\beta -\alpha} {\nabla} \varphi,  \\
    \textcolor{black}{\Hess^D \varphi(X,Y)} & \textcolor{black}{:=g(D_X (\nabla^D \varphi),Y)}\\
    & =V^{\beta - \alpha }\left(\Hess\, \varphi + \beta \d u \otimes \d \varphi + \beta \d \varphi \otimes \d u + \alpha g({\nabla} u, {\nabla} \varphi) {g}\right)\textcolor{black}{(X,Y)},\\
    {\Delta}^D \varphi &:=\operatorname{tr}\left(\Hess^D \varphi\right) = V^{\beta - \alpha}\left\{{\Delta} \varphi + (n \alpha + 2 \beta) g({\nabla} u, {\nabla} \varphi)\right\} .
\end{align*}
\textcolor{black}{Let $\Omega$ be a compact set in $M$ and $\nu$ be the outward unit normal vector field on $\partial \Omega$.
On $\partial \Omega$,
we set the second fundamental form $\mathrm{II}_{\partial \Omega}(X,Y) := g(\nabla_X \nu, Y)$ and the mean curvature $H_{\partial \Omega} := \mathrm{tr} \,\mathrm{II}_{\partial \Omega}$.}
We define the \textit{second fundamental form $\mathrm{II}_{\partial \Omega}^D$} and the \textit{mean curvature $H^D_{\partial \Omega}$ with respect to $D$} by
\begin{align}\label{def:affine-mean-curvature}
    \mathrm{II}_{\partial \Omega}^D &:= \mathrm{II}_{\partial \Omega} - \beta u_\nu \,g_{\partial \Omega}, \quad\,\, H^D_{\partial \Omega} := H_{\partial \Omega} + (n-1) \alpha u_\nu,
\end{align}
where \textcolor{black}{$g_{\partial \Omega}$ is the metric on $\partial \Omega$ induced from $g$ and $u_\nu := g(\nabla u, \nu)$}.
\textcolor{black}{
A hypersurface $\Sigma$ in $M$ is called \textit{$D$-minimal} if $H_\Sigma^D$ vanishes identically on $\Sigma$, 
where $H_\Sigma^D$ is defined in the same manner as in \eqref{def:affine-mean-curvature} for a unit normal vector field $\nu$ on $\Sigma$.
}
\textcolor{black}{We denote the Riemannian volume measure of $g_{\partial \Omega}$ by $v_{g,\partial \Omega}$ and $D$-gradient on $(\partial \Omega, g_{\partial \Omega})$ by $\nabla^D_{\partial \Omega}$.}
In \cite[Theorem 1.1]{LX}, 
they obtained the following Reilly formula:
\begin{align}\label{eq:d-reilly}
    & \quad\,\,\int_\Omega V^\tau \left(({\Delta}^D \varphi)^2-\left| \Hess^D  \varphi\right|^2- {\Ric}^D\left({\nabla}^D \varphi, {\nabla}^D \varphi\right) \right) \ \d v_g\\
    &=  \int_{\partial \Omega} V^\tau \left(H^D_{\partial \Omega}\, g\left({\nabla}^D \varphi, \nu\right)^2+\mathrm{II}_{\partial \Omega}^D\left(\nabla_{\partial \Omega}^D\, \psi, \nabla_{\partial \Omega}^D \,\psi\right)-2 V^{-\beta} g\left(\nabla_{\partial \Omega}^D\, \psi, \nabla_{\partial \Omega}^D\left(V^\beta \varphi_\nu\right)\right)\right)\ \d v_{g,\partial \Omega}\nonumber ,
\end{align}
where $\psi := \varphi|_{\partial \Omega}$,
$\tau := (n + 1)\alpha + \beta$,
\textcolor{black}{
We also have the following integration by parts formula (see e.g., \cite{LX,HBM}):
\begin{align}\label{eq:integration-by-parts}
    \int_\Omega V^\tau \varphi_1 \Delta^D\varphi_2 \ \d v_g = \int_{\partial \Omega} V^\tau \psi_1 g\left( \nabla^D\varphi_2, \nu \right)\ \d v_{g,\partial \Omega} - \int_\Omega V^\tau g \left( \nabla \varphi_1, \nabla^D \varphi_2 \right)\ \d v_g
\end{align}
for $\varphi_1, \varphi_2 \in C^\infty(\overline{\Omega})$ and $\psi_1 := \varphi_1|_{\partial \Omega}$.
}
\textcolor{black}{
We are now in a position to prove Theorem \ref{thm:choi-wang}.
}
\begin{proof}[\textcolor{black}{Proof of Theorem \ref{thm:choi-wang}}]
We apply the argument in \cite[Theorem 2]{CW}.
\textcolor{black}{
By Theorem \ref{thm:bochner},
we have $H^1(M;\mathbb{R}) = 0$,
which implies $H_{n-1}(M;\mathbb{R}) = 0$.
Combining this with the following exact sequence:
\begin{align*}
    \cdots \rightarrow H_n(M,\Sigma; \mathbb{R}) \rightarrow H_{n-1}(\Sigma ; \mathbb{R}) \rightarrow H_{n-1}(M; \mathbb{R}) \rightarrow \cdots,
\end{align*}
it follows that $H_n(M,\Sigma;\mathbb{R}) \rightarrow H_{n-1}(\Sigma;\mathbb{R})$ is surjective.
Therefore,
we see that $\Sigma$ divides $M$ into two components.
}

Let $\Omega$ be one component of these two components.
We note that $\Sigma = \partial \Omega$.
Let $\psi$ be the first eigenfunction of $\Delta^D_{\Sigma}$ such that $\Delta^D_{\Sigma}\, \psi + \lambda_{1} \psi = 0$ on $\Sigma$,
where we denote $\lambda_1 := \lambda_1(\Delta_\Sigma^D)$ for brevity of notations \textcolor{black}{(see Remarks \ref{rem:PDE} and \ref{rem:Laplace} below)}.
Also, 
let $\varphi$ be the solution of the following boundary value problem:
\begin{align}\label{eq:choi-wang-boundary-problem}
    \begin{cases}
        \Delta^D \varphi=0 & \text { on } \Omega, \\
        \varphi =\psi  & \text { on }   \partial  \Omega .
    \end{cases}
\end{align}
Let $V := \e^u$ and $n$ be the dimension of $M$.
By applying the Reilly formula \eqref{eq:d-reilly},
we see
\begin{align}\label{eq:choi-wang-2}
    0 \geq K \int_\Omega V^{n\alpha + 2\beta} |\nabla \varphi|^2 \ \d v_g + \int_{\partial \Omega}V^\tau \left(\mathrm{II}^D_{\partial \Omega}(\nabla^D_{\partial \Omega} \,\psi, \nabla^D_{\partial \Omega} \,\psi) - 2 V^{-\beta} g( \nabla^D_{\partial \Omega} \,\psi, \nabla^D_{\partial\Omega}(V^\beta \varphi_\nu) )\right)\ \d v_{g, \partial \Omega},
\end{align}
where $\nu$ is the outer unit normal vector field on $\partial \Omega$.
Here,
we have
\textcolor{black}{ 
\begin{align*}
    \int_{\partial \Omega} V^{\tau - \beta}g ( \nabla^D_{\partial \Omega} \,\psi, \nabla^D_{\partial \Omega}(V^\beta \varphi_\nu) )\ \d v_{g, \partial \Omega} 
    &= \int_{\partial \Omega} V^{n\alpha + \beta} g\left( \nabla_{\partial \Omega} \,\psi, \nabla^D_{\partial \Omega} \left( V^\beta \varphi_\nu\right) \right)\ \d v_{g,\partial \Omega}\\
    &= -\int_{\partial \Omega} V^{n\alpha + 2\beta} \varphi_\nu \,\Delta_{\partial \Omega}^D \,\psi\ \d v_{g,\partial \Omega}\\
    &= \lambda_1\int_{\partial\Omega} V^{n\alpha + 2\beta} \varphi_\nu\, \psi\ \d v_{g, \partial \Omega},
\end{align*}
where we used the integration by parts \eqref{eq:integration-by-parts} on $\Sigma$ in the second equality.}
\textcolor{black}{
Using \eqref{eq:integration-by-parts} and \eqref{eq:choi-wang-boundary-problem},
the last term is calculated as follows:
\begin{align*}
    \lambda_1\int_{\partial\Omega} V^{n\alpha + 2\beta} \varphi_\nu\, \psi\ \d v_{g, \partial \Omega} 
    &= \lambda_1 \int_{\partial \Omega} V^\tau \psi g\left( \nabla^D\varphi, \nu \right)\ \d v_{g,\partial \Omega}\\
    &= \lambda_1 \int_{\Omega} V^\tau g \left( \nabla\varphi, \nabla^D\varphi \right)\ \d v_g.
\end{align*}
}
Combining this with \eqref{eq:choi-wang-2}, 
we have 
\begin{equation*}
    0\geq (K - 2\lambda_1) \int_\Omega V^{n\alpha + 2\beta} |\nabla \varphi|^2 \ \d v_g + \int_{\partial \Omega} V^\tau \mathrm{II}^D_{\partial\Omega} (\nabla_{\partial \Omega}\,\textcolor{black}{\psi}, \nabla_{\partial \Omega}\,\textcolor{black}{\psi})\ \d v_{g,\partial\Omega}.
\end{equation*}
Since the sign of $\mathrm{II}_{\partial\Omega}^D$ changes depending on which component we choose as $\Omega$,
we may assume 
\begin{align*}
    \int_{\partial \Omega} V^\tau \mathrm{II}^D_{\partial\Omega} (\nabla_{\partial \Omega}\,\textcolor{black}{\psi}, \nabla_{\partial \Omega}\,\textcolor{black}{\psi})\ \d v_{g,\partial\Omega}\geq 0.
\end{align*}
Therefore, we have 
\begin{align*}
    0\geq (K - 2\lambda_1) \int_\Omega V^{n\alpha + 2\beta} |\nabla \varphi|^2 \ \d v_g.
\end{align*}
We complete the proof.
\end{proof}

\begin{remark}\label{rem:PDE}
\textcolor{black}{
We remark on $\lambda_1(\Delta^D_\Sigma), \psi$ and $\varphi$ in the proof above.
The existence of $\lambda_1(\Delta^D_\Sigma) > 0$ and $\psi\in W^{1,2}(\Sigma, v_{g,\Sigma})$ follows from a standard argument (see e.g., Schm\"{u}dgen \cite{konrad}) since we have the integration by parts formula for $\Delta^D_\Sigma$.
Indeed,
the formula \eqref{eq:integration-by-parts} on $\Sigma$ is written as follows: 
\begin{align}\label{eq:ibp-sigma}
    \int_\Sigma V^{n\alpha + \beta} \psi_1 \Delta^D_\Sigma\, \psi_2 \ \d v_{g,\Sigma} = - \int_\Sigma V^{n\alpha + \beta} g_\Sigma (\nabla_\Sigma \,\psi_1, \nabla_\Sigma^D\, \psi_2)\ \d v_{g,\Sigma}
\end{align}
for $\psi_1, \psi_2 \in C^\infty(\Sigma)$.
Inspired by this,
we set a bilinear form $\mathfrak{t}$ as follows:
\begin{align*}
    \mathfrak{t}(\psi_1, \psi_2) := \int_\Sigma V^{n\alpha + \beta} g_\Sigma(\nabla_\Sigma\, \psi_1, \nabla^D_\Sigma \,\psi_2)\ \d v_{g,\Sigma} + \int_\Sigma  V^{n\alpha + \beta}\,\psi_1 \psi_2 \ \d v_{g,\Sigma}
\end{align*}
for $\psi_1,\psi_2 \in W^{1,2}(\Sigma,v_{g,\Sigma})$.
Since $\Sigma$ is compact,
$V$ is bounded.
In particular,
$\mathfrak{t}$ is continuous and coercive with respect to the $W^{1,2}$-norm.
Here,
we note that the norms of $W^{1,2}(\Sigma,v_{g,\Sigma})$ and $W^{1,2}(\Sigma,V^{n\alpha + \beta}v_{g,\Sigma})$ are equivalent,
where $W^{1,2}(\Sigma,V^{n\alpha + \beta}v_{g,\Sigma})$ is the weighted Sobolev space equipped with the following norm:
\begin{align*}
    \|\phi\|_{W^{1,2}(\Sigma,V^{n\alpha + \beta}v_{g,\Sigma})} := \left(\int_\Sigma V^{n\alpha + \beta} \left(\phi^2 + |\nabla_\Sigma \,\phi|^2\right)  \ \d v_{g,\Sigma} \right)^{\frac{1}{2}} .
\end{align*}
Furthermore,
the embedding
\begin{align*}
    W^{1,2}(\Sigma,v_{g,\Sigma}) \hookrightarrow L^2(\Sigma, V^{n\alpha + \beta}v_{g,\Sigma})
\end{align*}
is compact.
Therefore,
we may apply \cite[Theorem 11.3]{konrad} to the bilinear form $\mathfrak{t}$.
Together with \eqref{eq:ibp-sigma},
we see that $\Delta^D_\Sigma$ has discrete spectrum.
This leads us to the existence of $\lambda_1(\Delta^D_\Sigma) > 0$ and $\psi\in W^{1,2}(\Sigma, v_{g,\Sigma})$ (see also \cite[Proposition 5.12]{konrad}).}

\textcolor{black}{
The smoothness of $\psi$ follows from standard elliptic theory (see e.g., Gilbarg-Trudinger \cite{GT}).
More precisely,
by \cite[Lemma 3.4]{LX},
$\Delta^D_\Sigma$ can be written in divergence form as follows:
\begin{align*}
    \Delta^D_\Sigma \,\psi = V^{-(n\alpha + \beta)} \mathrm{div}_{\Sigma}\left( V^{n\alpha + \beta}\nabla^D_\Sigma \,\psi \right),
\end{align*}
where $\mathrm{div}_\Sigma$ denotes the divergence on $(\Sigma, g_\Sigma)$.
Hence,
we may apply \cite[Corollary 8.11]{GT} to our setting.
This yields the smoothness of $\psi$.
Lastly,
we note that the existence and smoothness of the solution $\varphi$ of \eqref{eq:choi-wang-boundary-problem} follow by applying \cite[Theorem 6.14]{GT}.}
\end{remark}
\begin{remark}\label{rem:Laplace}
\textcolor{black}{
There is another perspective on the argument in Remark \ref{rem:PDE}.
Let $(M,g)$ be an $n$-dimensional Riemannian manifold,
$u\in C^\infty(M)$ and $\alpha, \beta \in \mathbb{R}$.
We set $D := \nabla^{u,\alpha,\beta}$ and $\tau := (n + 1)\alpha + \beta$.
Then,
a direct calculation shows that $\Delta^D$ is the weighted Laplacian on the weighted Riemannian manifold $(M,\e^{(\alpha - \beta)u}g, \e^{\tau u}v_g)$.
Hence,
the existence of $\lambda_1(\Delta^D_\Sigma), \psi$ and $\varphi$ in the proof of Theorem \ref{thm:choi-wang} can also be obtained by applying the argument in the weighted setting to our case.}
\end{remark}

\begin{remark}
We refer to Choi-Wang \cite[Theorem 2]{CW} for the case $u\equiv 0$.
The argument in Choi-Schoen \cite{CS},
where they extend the Choi-Wang inequality in \cite{CW} to a broader class of manifolds,
cannot be directly adapted in our setting.
In their argument,
both the Frankel theorem and the finiteness of the fundamental group are essential.
In our setting, 
we can apply the argument in Li-Wei \cite[Lemma 5]{LW} using the Reilly formula \eqref{eq:d-reilly},
and this leads to the conclusion that a Frankel type theorem holds in our case as well (see also \cite[Remark 3.8]{FS} for the case of weighted Riemannian manifolds).
However,
regarding the finiteness of the fundamental group,
there is still a difficulty in obtaining it in our setting.
As mentioned in the introduction,
this issue is left for future work.
\end{remark}

\subsection*{{\textbf{Acknowledgements}}}
The author would like to express sincere gratitude to Ryu Ueno for fruitful discussions.
\textcolor{black}{The author is grateful to Professors Norihisa Ikoma and Yohei Sakurai for their valuable comments.}
This work was supported by JSPS KAKENHI Grant Number JP25KJ0271.

\bibliographystyle{amsplain}
\bibliography{ref} 
\end{document}